\documentclass[11pt, reqno]{amsart}
\usepackage{amssymb,latexsym,amsmath,amsfonts}
\usepackage[mathscr]{eucal}

\usepackage{color}
\definecolor{Red}{rgb}{1,0,0}
\definecolor{Blue}{rgb}{0,0,1}

\numberwithin{equation}{section}

\theoremstyle{definition}
\newtheorem{definition}{Definition}[section]

\theoremstyle{remark}
\newtheorem{remark}[definition]{Remark}

\theoremstyle{plain}
\newtheorem{theorem}[definition]{Theorem}
\newtheorem{result}[definition]{Result}

\newtheorem{proposition}[definition]{Proposition}

\newcommand{\zt}{\zeta}

\newcommand{\ent}{\epsilon_0}

\newcommand{\bdy}{\partial}
\newcommand{\OM}{\Omega}
\newcommand{\dee}{\mathbb{D}}

\newcommand{\smoo}{\mathcal{C}}
\newcommand{\holo}{\mathcal{O}}

\newcommand{\bcdot}{\boldsymbol{\cdot}}
\newcommand{\bdot}{\boldsymbol{.}}
\newcommand\newfrc[2]{{}^{\raisebox{-2pt}{$\scriptstyle {#1}$}}\!/_{\raisebox{2pt}{$\scriptstyle {#2}$}}}

\newcommand{\CC}{\mathbb{C}^2}
\newcommand{\cplx}{\mathbb{C}} 
\newcommand{\Cn}{\mathbb{C}^n}

\begin{document}

\title[A simple Kontinuit{\"{a}}tssatz]{A simple Kontinuit{\"{a}}tssatz}
\author{Chandan Biswas}

\thanks{This work is supported by a UGC Centre for Advanced Study grant and by a scholarship from the IISc}

\address{Department of Mathematics, Indian Institute of Science, Bangalore 560012, India}
\email{biswas.270@gmail.com}


\date{\today}

\begin{abstract} 
We are interested in several informal statements referred  as ``Kontinuit{\"{a}}tssatz'' in 
the recent literature on analytic continuation. The basic (unstated) principle that seems to be 
in use in these works appears to be a folk theorem. We provide a precise statement of this 
folk Kontinuit{\"{a}}tssatz and give a proof of it.
\end{abstract}
\maketitle

\section{Introduction and statement of results}\label{S:intro}
In several recent papers on the subject of analytic continuation --- see 
\cite{Chirka01,CRosay01,Bb01}, for instance --- one comes across versions of an 
informal lemma referred as ``Kontinuit{\"{a}}tssatz" or ``the continuity principle". 
The various conclusions deduced from it are obtainable from the following general 
principle:

\begin{itemize}
\item[$(*)$] Let $\OM$ be a domain in $\Cn, \ n\geq 2$, let $\{D_{t}\}_{t\in [0,1]}$ 
be an indexed family of smoothly bounded open sets in $\cplx$, and let 
$\Psi_{t}:\overline{D}_{t}\rightarrow\mathbb{C}^{n-1}$ be maps such that 
$\Psi_{t}\in H(D_{t})\cap\smoo(\overline{D}_{t})$ for each $t\in [0,1]$. 
Assume that the $D_{t}$'s vary continuously with $t$ in some appropriate 
sense such that $\Gamma:=\cup_{t\in [0,1]}(\overline{D}_{t}\times\{t\})$ forms 
a ``nice'' compact body: such that $\overline{(\cplx\times (0,1))\cap \partial\Gamma}$ is 
a smooth submanifold with boundary, for instance. Also assume that 
$F:\Gamma\rightarrow \mathbb{C}^{n-1}$ given 
by $F(\zt,t):=\Psi_{t}(\zt)$ is continuous. Suppose we know that:

\begin{itemize}
 \item[$i)$] $\{(\zt,\Psi_{t}(\zt)):\zt\in \bdy D_{t}\}\subset K$, for some compact set 
 $K\subset\OM$ for each $t\in [0,1]$;
 \item[$ii)$] ${\sf Graph}(\Psi_{0})\subset\OM$ and ${\sf Graph}(\Psi_{1})\nsubseteq\OM$.
\end{itemize}

\noindent{Then there exists a neighbourhood $V$ of ${\sf Graph}(\Psi_{1})\cup K$ and 
a neighbourhood $W$ of $K$, $K\Subset W\subset \OM\cap V$, such that for 
each $f\in H(\OM)$ there exists a $G_{f}\in H(V)$ and $G_{f}|_{W}\equiv f|_{W}$.}
\end{itemize}
\smallskip

\noindent{Here, if $D$ is an open set then $H(D)$ denotes the class of holomorphic functions 
 or maps on $D$.}
\smallskip

The statement $(*)$ {\em per se} is not new. But, to the best of my knowledge, {\em there does not 
seem to be a proof of $(*)$, in that generality, in the literature} (also, the reader is directed 
to Remark~\ref{R:Chirka-Stout} below). Of course, $(*)$ is easy to derive in special cases. If 
${\sf Graph}(\Psi_{t})\cap{\sf Graph}(\Psi_{s})=\emptyset$ for $s\neq t$, 
then $(*)$ is classical. If $D_{t}=D$, a fixed planar region, for each 
$t\in[0,1]$, then the above can be deduced by analytically continuing an $H(\OM)$-germ 
(see Section~\ref{S:notation} for a definition) along the path 
$t\mapsto(\zt,\Psi_{t}(\zt))$ for a fixed $\zt\in D$. When the $D_t$'s vary, 
but still assuming ${\sf Graph}(\Psi_{t})\cap{\sf Graph}(\Psi_{s})=\emptyset$ for 
$s\neq t$, $(*)$ follows from Behnke's work \cite{Behnke01}.
\smallskip

When ${\sf Graph}(\Psi_{t})\cap{\sf Graph}(\Psi_{s})\neq\emptyset$ for some 
$t\neq s$, we must worry about multivaluedness. This motivates us to ask whether the 
maps ${\sf id}_{\overline{D}_{t}}\!\times\Psi_{t}$ can be ``lifted" into the 
envelope of holomorphy of $\OM$ continuously in the parameter $t\in [0,1]$. 
Intuitively, this seems possible. However, note that as $D_t$ varies with $t$, the 
$D_t$'s do not even have to be conformally equivalent to each other. In this generality, 
{\em rigorously} producing such a ``lifting'' is not entirely effortless. It would thus 
be nice to have an argument written up for this. 
\smallskip

To summarise: it appears that $(*)$ in the above generality is folk theorem (the remark 
following Theorem~\ref{T:MainThm} is relevant here). The aim of this work is to give a proof 
of $(*)$ (or rather, a slight generalization of it) after making precise some of the 
loosely-stated assumptions in $(*)$.
\smallskip

Some notation: if $S\subset\cplx\times[0,1]$, we shall use $\,^t{S}$ to denote the set
$\{\zt\in \cplx : (\zt,t)\in S\}$.

\begin{theorem}\label{T:MainThm}
Let $\OM$ be a domain in $\Cn$, $n\geq 2$. Let $\Gamma$ be a compact body in $\cplx\times[0,1]$
such that $\partial_*\Gamma:=\overline{(\cplx\times (0,1))\cap \partial\Gamma}$ is a 
(not necessarily connected) $\smoo^2$-smooth submanifold with boundary. Assume that the sets
$$
 D_t := \begin{cases}
	\,^t(\Gamma^{\circ}), &\text{if $t\in (0,1)$}, \\
	(\,^t\Gamma)\cap\left(\cplx\setminus{\,^t(\partial_*\Gamma)}\right), &\text{if $t=0,1$}, 
	\end{cases}
$$
are domains in $\cplx$ with $\smoo^2$-smooth boundaries for each $t\in [0,1]$, and that 
the set-valued function $[0,1]\ni t\longmapsto\overline{D}_t$ is continuous relative to the 
Hausdorff metric (on the space of compact subsets of $\cplx$). Let $K$ be a compact subset
of $\OM$ and suppose there is a 
continuous map $\Psi:\Gamma\rightarrow\Cn$ with the following properties:
\begin{itemize}
 \item[$1)$] For each $t\in[0,1]$, $\Psi_t:=\Psi(\bcdot,t)\in H(D_t)\cap\smoo(\overline{D}_t)$;
 \item[$2)$] $\cup_{t\in [0,1]}\Psi_t(\partial D_t) = K$  and
 $\Psi_t(D_t)\cap K = \emptyset$ for each $t\in [0,1]$;
 \item[$3)$] $\Psi_0(\overline{D}_0)\subset \Omega$ and $\Psi_1(\overline{D}_1)\nsubseteq\OM$;
 \item[$4)$] For each $t\in[0,1]$, $\Psi_t$ is a continuous imbedding 
 and $\Psi_t|_{D_t}$ is a holomorphic imbedding.
\end{itemize}
Then there exists a neighbourhood $V$ of $\Psi_1(\overline{D}_1)\cup K$ and
a neighbourhood $W$ of $K$ satisfying $K\Subset W\subset \OM\cap V$ such that 
for any $f\in H(\OM)$, there exists $G_f\in H(V)$ such that
$G_f|_{W}\equiv f|_{W}$.
\end{theorem}

\begin{remark}\label{R:Chirka-Stout} 
Theorem~\ref{T:MainThm} would follow as a special case of a 
Kontinuit{\"{a}}tssatz-type result by Chirka and Stout \cite{Chirka-Stout}.
However, there are gaps in their proof. From Theorem 3 of the paper \cite{PJ01} by
J{\"o}ricke and Porten, one can construct an example showing that the
Chirka--Stout Kontinuit{\"{a}}tssatz is erroneous. In his thesis,  L. Nobel \cite{LNobel}
has given an alternative to the Chirka--Stout statement. However, the proof of
Nobel's (unpublished) result requires a considerable amount of machinery. 
Thus, an elementary proof in the set-up of Theorem~\ref{T:MainThm} would be
desirable. We note that in the very special case when $D_t = \dee$
(the open unit disc in $\cplx$) for every $t$, the assertion of Chirka--Stout can be deduced
without much machinery. We refer the reader to \cite{LNobel12} by Nobel for a precise
statement in this simple set-up. 
\end{remark}

Two observations on the geometry of $\Gamma$ in Theorem~\ref{T:MainThm}
are in order. The assumption that the boundary components of the slices $D_t$ are
$\smoo^2$-smooth curves is to alert the reader to the fact that we require 
$\partial D_t$ to have no isolated points, i.e. punctures in $D_t$, or components with 
nodal singularities (whose appearance in some $D_{t^0}$ would
indicate a change in the homtopy type of the $D_t$'s as $t$ varies in a neighbourhood of $t^0$).
With this being clear, we could have stated that the $D_t$'s are domains with
$\smoo^\infty$-smooth boundaries with no loss of generality. Note that the hypotheses of
Theorem~\ref{T:MainThm} do not preclude $\Psi_s(D_s)\cap \Psi_t(D_t)\neq\emptyset$
for some $t\neq s$. This actually happens in the papers cited above. 
We therefore have to control against multivaluedness --- which would require
{\em considerably more} machinery than is used in Section~\ref{S:funda} (and much, much
lengthier arguments) if, for instance, the $D_t$'s were allowed to have punctures. Secondly:
we point out to the reader how easy it is in the case of \cite{Bb01} to choose a compact
body $\Gamma$ --- wherever a Kontinuit{\"a}tssatz is invoked in \cite{Bb01} --- that has the
geometry given in Theorem~\ref{T:MainThm} and delivers the authors' conclusion.
\smallskip
 
In this proof, we shall use the notation $(\widetilde{\OM},p,j)$ to denote
the envelope of holomorphy of $\OM$, where $p$ is the canonical local 
homeomorphism into $\Cn$ that gives $\widetilde{\OM}$ the structure of an 
unramified domain spread over $\Cn$ and $j:\OM\hookrightarrow\widetilde{\OM}$
is an imbedding such that $p\circ j\equiv {\sf id}_{\OM}$. The only machinery 
that our proof of Theorem~\ref{T:MainThm} requires is Thullen's
construction of the envelope of holomorphy of $\OM$ 
\cite[Chapter 6, Theorem 1]{RNM01}. All other elements of our proof are
elementary and developed from scratch. 
\smallskip

Now, the usefulness of the Kontinuit{\"{a}}tssatz is that it gives us information 
about $p(\widetilde{\OM})$. Generally, we have very little information about how
$\widetilde{\OM}$ looks. But, if we have explicit information about the pair
$(\OM,K)$ occuring in Theorem~\ref{T:MainThm}, then we gain information
about $p(\widetilde{\OM})$. In our proof of Theorem~\ref{T:MainThm}, an
{\em explicit} description of $V$, in terms of 
$\OM,K$ and $\Psi_1(\overline{D}_1)$, is obtained.
\smallskip

If we have slightly more information about $(\widetilde{\OM},p,j)$ than is 
available in Theorem~\ref{T:MainThm}, then it should be possible to relax 
some of the conditions on the family $\{\Psi_t\}_{t\in[0,1]}$. For example: could 
we drop the requirement that $\Psi_t,t\in[0,1]$, be imbeddings\,? This leads to our 
next theorem.

\begin{theorem}\label{T:thm2}
Let $\OM$ be a domain in $\CC$, and let $(\widetilde{\OM},p,j)$ denote 
the envelope of holomorphy of $\OM$. Let $\Gamma$ be a compact body having exactly
the same properties as in Theorem~\ref{T:MainThm}
and let the domains $D_t$, $t\in [0,1]$, be the same as in Theorem~\ref{T:MainThm}. 
Let $K$ be a compact subset of $\OM$ and
suppose there is a continuous map $\Psi:\Gamma\rightarrow \CC$ satisfying 
$(1)$, the first part of $(2)$, and $(3)$ of Theorem~\ref{T:MainThm}.
Let us assume that $p$ is injective on $p^{-1}(\OM)$. Then there 
exists an open set $V\supset\Psi_{1}(\overline{D}_{1})\cup \OM$ such that for each 
$f\in H(\OM)$, there exists $G_f\in H(V)$ such that $G_f|_{\OM}\equiv f$.
\end{theorem}

Both theorems require a description of Thullen's construction of the 
envelope of holomorphy. This is provided in Section~\ref{S:notation}.
Also, we require a lifting result for the family $\{\Psi_t\}_{t\in[0,1]}$.
We state and proof this result in Section~\ref{S:funda}.
\medskip

\section{Notations and terminology}\label{S:notation}

We begin this section by observing that if $\OM$ is a domain in $\Cn$, then the envelope 
of holomorphy always exists. This was shown by Thullen --- see Chapter~6, Theorem~1 in 
\cite{RNM01}. Thullen's construction depends on the sheaf of $S$-germs on $\Cn$, denoted here 
as $\holo(S)$. Given a point $a\in \Cn$, an {\em $S$-germ at $a$} is the equivalence class
of the objects $(U,\phi_s:s\in S)$, where $U$ is an open set containing $a$; $\{\phi_s:s\in S\}$
is a family of functions, holomorphic on $U$, that are indexed by $S$; and
\begin{align}
 (U,\phi_s:s\in S)\,\thicksim\,(V,\psi_s:s\in S) \ \iff \ &\text{there exists a neighbourhood $W$ of $a$,} 
										\notag \\
 &\text{$W\subset U\cap V$ such that $\left.\phi_s\right|_W\equiv \left.\psi_s\right|_W \ 
 \text{for all}\,\, s\in S$.} \notag
\end{align}
The collection of all $S$-germs at $a$ is denoted by $\holo_a(S)$, and $\holo(S):= 
\cup_{a\in \Cn}\holo_a(S)$. We refer the reader to the first few pages of Chapter~6 of \cite{RNM01}
for the definition of $\holo(S)$ and the topology upon it that turns $(\holo(S),p,\Cn)$ 
into an unramified domain spread over $\Cn$ (here $p$ denotes the sheaf projection, i.e.
$p(x)=a\,\iff\,x\in \holo_a(S)$). 
\smallskip

We note here that since the letter $\holo$ will denote certain types of sheaves, we shall 
use $H(\OM)$ to denote the class of holomorphic functions on the domain $\OM$.
\smallskip

The case that is important for us here is when $S=H(\OM)$. Thullen's proof 
of the existence of the envelope of holomorphy $(\widetilde{\OM},p,j)$ 
is to realise this object as a subset of the sheaf $\holo(H(\OM))$. We summarise this 
construction as follows. {\bf Note:} throughout this paper, if 
$(U,\phi_f:f\in H(\OM))$ is a representative of an $H(\OM)$-germ at $a$, we shall 
denote that germ as $[U,\phi_f:f\in H(\OM)]_a$. We shall denote the 
canonical extension of $f$ to $\widetilde{\OM}$, $f\in H(\OM)$, as $F_{f}$.
\begin{itemize}
 \item[$(i)$] 
 $\widetilde{\OM}$ is realised as the connected component of $\mathcal{O}(H(\OM))$ (i.e. the 
 sheaf of $H(\OM)$-germs of holomorphic functions) containing the (connected) set
 $\{[\OM, f:f\in H(\OM)]_a: a\in \OM\}.$
 \item[$(ii)$]\label{defini} 
 Given any $f\in H(\OM)$, the function $F_f$ is defined on $\widetilde{\OM}$
 as follows. Each point $x\in \widetilde{\OM}$ is an 
 $H(\OM)$-germ, so there exists a neighbourhood $U$, containing 
 $a=p(x)\in \Cn$, such that 
 $x=\left[U, \phi_g:g\in H(\OM)\right]_a.$ Then $F_f(x):=\phi_f(a).$
\end{itemize}
\smallskip

The next two definitions are pertinent on any unramified domain 
$(X,p_0,\Cn)$ spread over $\Cn$. In what
follows, the notation $D^{n}(a,r)$ will denote the polydisc
$$
D^{n}(a,r) := \{z\in \Cn: |z_j-a_j|<r, \ j=1,\dots,n\}
$$
centered at the point $a=(a_1,\dots,a_n)$.

\begin{definition}\label{polyd}
Let $(X,p_0,\Cn)$ be an unramified domain spread over $\Cn$. Let $a\in X$. A 
{\em{polydisc of radius $r$ about $a$}} is a connected open set 
$U,\,a\in U$, such that $p_0|_{U}$ is an analytic isomorphism onto the set 
$D^{n}(b,r)$; here $b=p_0(a)=(b_1,b_2\ldots,b_n)$. We denote the set $U$ by $P(a,r)$. 
The {\em{maximal polydisc}} about $a$ is the union of all polydiscs about $a$. 
\end{definition}

As suggested by the last sentence above, the maximal polydisc is itself a polydisc about 
$a$, and its radius
$$
 r_0 = \sup\{r>0: P(a,r) \ \text{is a polydisc about $a$}\}.
$$
See \cite[Chapter 7, Lemma 4]{RNM01} for details.

\begin{definition}
Let $(X,p_0,\Cn)$ be an unramified domain spread over $\Cn$. Let $a\in X$.
The radius of the maximal polydisc about $a$ is called the {\em{distance of 
$a$}} from the boundary of $X$ and is denoted by $d_{X}(a)$. If $A\subset X$, 
we set
$$
d_{X}(A):=\text{inf}_{a\in A}d_{X}(a)
$$
\end{definition}
\medskip

\section{A fundamental proposition}\label{S:funda}

In this section, we shall prove a proposition that allows us to tackle both 
Theorem~\ref{T:MainThm} and Theorem~\ref{T:thm2} within the same framework. 
We point out to the reader that an essential part of the proof of the 
proposition below depends on the structure of Thullen's construction of the 
envelope of holomorphy $(\widetilde{\OM},p,j)$. The other important fact
that we use is the Cartan--Thullen characterisation of a domain of holomorphy.

\begin{result}[Cartan--Thullen]\label{Car-Thu}
Let $(X,p_0,\Cn)$ be an unramified domain over $\Cn$. If $X$ is a domain of holomorphy, then,
for any compact $K\Subset X$, $d_X(K)=d_X(\widehat{K}_{X})$.
\end{result}

\noindent{In the above result, $\widehat{K}_{X}$ denotes the {\em holomorphically convex hull} of $K$.
Recall that
$$
 \widehat{K}_{X} := \{x\in X: |f(x)|\leq \sup\nolimits_K|f| \;\, \text{for all}\,\, f\in H(X)\}.
$$}

We make one small notational point. In this section, if $z=(z_1,\dots,z_n)\in \Cn$, then
$$
 \|z\| := \max\{|z_1|,\dots,|z_n|\}.
$$

\begin{proposition}\label{KSTZ3}
Let $\OM$ be a domain in $\Cn$, $n\geq 2$, and let $(\widetilde{\OM},p,j)$ denote the 
envelope of holomorphy of $\OM$. Let $K$ be compact subset of $\OM$. Let $\Gamma$ 
be a compact body in $\mathbb{C}\times[0,1]$ having exactly the same 
properties as in Theorem~\ref{T:MainThm}. Let $\Psi:\Gamma\rightarrow\Cn$ be 
a continuous map satisfying conditions $(1)$, the first part of $(2)$, and $(3)$ from
Theorem~\ref{T:MainThm}. Then, for each $t\in [0,1]$, there exists a continuous map 
$\tilde{\Psi}_{t}:\overline{D}_{t}\rightarrow 
\widetilde{\OM}$ such that $p\circ\tilde{\Psi}_{t}\equiv\Psi_{t}$ and satisfies 
$\tilde{\Psi}_{t}(\zt)=[\OM,f: f\in H(\OM)]_{\Psi_{t}(\zt)}$ for every $\zt\in \bdy{D_{t}}$.
\end{proposition}
\begin{proof}
Define the set
\begin{align} 
S := &\{t\in [0,1]: \exists\,\tilde{\Psi}_{t'}:\overline{D}_{t'}\rightarrow \widetilde{\OM}, \,\tilde 
{\Psi}_{t'} \in \smoo(\overline{D}_{t'}),\, 
p\circ\tilde{\Psi}_{t'}\equiv \Psi_{t'} \ \text{and} \notag \\ 
	& \tilde{\Psi}_{t'}(\zt)=[\OM,f: f\in H(\OM)]_{\Psi_{t'}(\zt)}\ \text{for all}\ \ 
\zt\in \bdy{D_{t'}}\ \ \text{and for all}\ \ t'\in[0,t]\}. \notag 
\end{align}
We shall show that $S$ is both open and closed in $[0,1]$ and $S\neq\emptyset$. To begin, 
define $\tilde{\Psi}_0:\overline{D}_0 \rightarrow \widetilde{\OM}$ by 
\begin{eqnarray}\label{defn 0}
\tilde{\Psi}_0(\zt):=[\OM,f:f\in H(\OM)]_{\Psi_0(\zt)}\ \text{for all}\ \ \zt\in \overline{D}_0.
\end{eqnarray}
Since $\Psi_0(\overline{D}_0)\subset\OM$, $(\ref{defn 0})$ is well-defined. Continuity is 
easy. Therefore $0\in S$.
\smallskip

Let $\tau=\sup S$. By the definition of $S$ we have that $[0,\tau)\subset S$. So to show that 
$S$ is closed it is therefore enough to show that $\tau\in S$. For this purpose, let us consider the 
envelope of holomorphy, $(\widetilde{\OM},p,j)$. We recall an important aspect of the construction 
of $\widetilde{\OM}$ (given at the beginning of Section~\ref{S:notation}). For each $f\in H(\OM)$, 
its canonical extension to $\widetilde{\OM}$ is described as follows: 
\begin{itemize}
\item[$(**)$] For each point $x\in \widetilde{\OM}$, which is an $H(\OM)$-germ, let 
$a=p(x)\in\Cn$ and let $U$ be a neighbourhood of $a$ such that 
$x=[U,\phi_g:g\in H(\OM)]_a$. Then $F_f(x)=\phi_f(a)=\phi_f(p(x))$.
\end{itemize}
\smallskip

\noindent{{\bf Claim:} {\em If $t\in S$, then $\tilde{\Psi}_t(\overline{D}_t)\subset 
(\widehat{K^*})_{\widetilde{\OM}}$ 
where $K^*=j(K)$.}}

\noindent{{\em Proof of claim:} Let $F\in H(\widetilde{\OM})$. As $p\circ\tilde{\Psi}_t=\Psi_t$, 
both $p$ and 
$\Psi_t$ are holomorphic, and $p$ is a local isomorphism (biholomorphism) we have 
$\tilde{\Psi}_t\in H(D_t)$ (by the definition of holomorphicity of maps with values in an 
unramified domain spread over $\Cn$). 
Therefore $F\circ\tilde{\Psi}_t\in H(D_t)\cap\smoo(\overline{D}_t)$. 
Let $\zt\in D_t$. By the maximum modulus principle 
\begin{eqnarray}\label{hatinclu}
|F(\tilde{\Psi}_t(\zt))|
= |F\circ\tilde{\Psi}_t(\zt)|
\leq \sup\nolimits_{\bdy{D_t}}|F\circ\tilde{\Psi}_t| 
= \sup\nolimits_{\tilde{\Psi}_t(\bdy{D_t})}|F|
\leq \sup\nolimits_{K^*}|F|,
\end{eqnarray}
because $\tilde{\Psi}_{t}(\bdy D_{t})\subset K^*$ (this is a consequence of property $(i)$, 
stated in Section~\ref{S:notation}, of the envelope of holomorphy). Thus 
$\tilde{\Psi}_t(\overline{D}_t)\subset (\widehat{K^*})_{\widetilde{\OM}}$. Hence the claim.}
\smallskip

For any $\zt\in\overline{D}_t$ and $t\in S$, let us denote 
\begin{eqnarray}\label{repn.}
\tilde{\Psi}_t(\zt)=[U^{t,\zt},\,\phi^{t,\zt}_f:f\in H(\OM)]_{\Psi_t(\zt)} 
\end{eqnarray}
where $U^{t,\zt}$ is some neighbourhood of $\Psi_t(\zt)$ and $\phi^{t,\zt}_f\in H(U^{t,\zt})$.
\smallskip

In what follows, we shall abbreviate $d_{\widetilde{\OM}}(K^*)$ to $d(K^*)$. There exists an 
$\epsilon >0$, {\em a priori} $\epsilon=\epsilon(t,\zt)$, such that all the $H(\OM)$-germs in 
the definition below makes sense and the open set 
\begin{eqnarray}\label{basicnbd}
\mathcal{N}(\epsilon,t,\zt):=\{[U^{t,\zt},\,\phi_{f}^{t,\zt}:f\in H(\OM)]_a 
:a\in D^n(\Psi_{t}(\zt),\epsilon)\}
\end{eqnarray}
is contained in $\widetilde{\OM}$. As 
$\widetilde{\OM}$ is a domain of holomorphy, it follows from the above claim and from the
Cartan--Thullen Theorem that $P(\tilde{\Psi}_t(\zt),d(K^*))\subset\widetilde{\OM}$  for all 
$t\in S$ and for all $\zt\in \overline{D}_t$. Here $P(\tilde{\Psi}_t(\zt),d(K^*))$ 
is as explained in Definition~\ref{polyd}. It follows from the observation following 
Definition~\ref{polyd} that $p|_{P(\tilde{\Psi}_t(\zt),\,d(K^*))}$ is invertible. Thus, the functions 
\begin{eqnarray}\label{formula}
\Phi_f^{t,\zt}:=F_f\circ {(p|_{P(\tilde{\Psi}_t(\zt),\,d(K^*))})}^{-1} 
\end{eqnarray}
are holomorphic functions on $D^n(\Psi_t(\zt), d(K^*))$. Given our description 
$(**)$ above of the construction of $F_f$, the above statement gives us the 
following equality of $H(\OM)$ germs:
$$
[U^{t,\zt},\,\phi_f^{t,\zt}: f\in H(\OM)]_{\Psi_t(\zt)}=
[D^n(\Psi_t(\zt), d(K^*)),\,\Phi_f^{t,\zt}:f\in H(\OM)]_{\Psi_t(\zt)}.
$$
We summarise this as the following:
\medskip

{\addtolength{\leftskip}{6mm}

\noindent{{\bf Fact 1:} For any $t\in S$ and $\zt\in \overline{D}_t$,  
 $(D^n(\Psi_t(\zt), d(K^*)),\,\Phi_f^{t,\zt}:f\in H(\OM))$ is a representative of the 
 $H(\OM)$-germ $\tilde{\Psi}_t(\zt)$, where $\Phi_f^{t,\zt}$ is the function given by 
 (\ref{formula}).}

}
\medskip

\noindent{Given this fact, the open set $\mathcal{N}(\epsilon,t,\zt)$ makes sense for any number 
$\epsilon$ such that $0 < \epsilon\leq d(K^*)$, and is contained in $\widetilde{\OM}$ for all 
$\zt\in\overline{D}_t$, for all $t\in S$. Furthermore, it will be understood that, for the remainder 
of this proof, {\em $U^{t,\zt}=D^n(\Psi_t(\zt), d(K^*))$ and $\phi_f^{t,\zt}=\Phi_f^{t,\zt}$ 
whenever referring to the representation \eqref{repn.} for $\tilde{\Psi}_t$ for $t\in S$.}}
\smallskip

As $\Psi:\Gamma\rightarrow \Cn$ is a continuous map and $\Gamma$ 
is compact, there exists $\delta>0$ such that for $(\zt,t_1), (\eta,t_2)\in \Gamma$,
\begin{eqnarray}\label{cont}
|(\zt, t_1)-(\eta, t_2)| < \delta \ \Rightarrow \ 
\|\Psi_{t_1}(\zt)-\Psi_{t_2}(\eta)\| < d(K^*), 
\end{eqnarray} 
where, for any $(\zt,t)\in \cplx\times[0,1]$, $|(\zt,t)|:=\sqrt{|\zt|^2+t^2}$.
\smallskip

In the following argument $D(a,R)$ will denote the disc in $\cplx$ with 
centre $a$ and radius $R$. For $t\in [0,1]$ and $\alpha > 0$, set $\mathcal{A}_{t,\alpha}:= 
\{\zt\in \cplx: {\rm dist}[\zt,\bdy{D_t}]<\alpha\}$. By our hypotheses: 
\begin{itemize}
 \item As $\bdy{D_t}$ is $\smoo^2$-smooth, for each $t\in [0,1]$ there exists an 
 $\epsilon > 0$, which depends on $t$, such that for each $z\in \mathcal{A}_{t,\epsilon}$ 
 there is a unique closest point 
 in $\bdy{D_t}$, and such that each connected component of $\bdy{D_t}$ is a strong 
 deformation retract of the component of $\mathcal{A}_{t,\epsilon}$ containing it. 
 Let $\epsilon^*=\epsilon^*(t)$ be the largest $\epsilon > 0$ for which $\mathcal{A}_{t,\epsilon}$ 
 has these properties. 
 \item By the properties of $\partial_*\Gamma$, given $\alpha > 0$ small, there exists a 
 $\delta_*=\delta_*(\alpha,t) > 0$ such that $\mathcal{A}_{t,\epsilon^*(t)-\alpha/2}\subset 
 \mathcal{A}_{s,\epsilon^*(s)}$ and $\partial D_s\subset \mathcal{A}_{t,\alpha/2}$ for 
 every $s\in [0,1]\cap [t-\delta_*,t+\delta_*]$.
\end{itemize}
From this, it follows that $\epsilon^*$ is lower semicontinuous. 
It now follows from compactness of $[0,1]$ and an elementary argument that there exists a 
constant $r$, $0<r\leq \delta$, such that for each $t\in [0,1]$ and 
each $\zt\in \overline{D}_t$, $D(\zt,r)\cap \overline{D}_t$ is connected. 
Let us now {\em fix} a $t\in S$. By continuity, there exists a constant $\delta_t$, 
$0<\delta_t\leq r$, depending only on $t$, such that 
$$
\eta\in D(\zt, \delta_t)\cap\overline{D}_{t} \ \Rightarrow \  
\tilde{\Psi}_t(\eta)\in \mathcal{N}(d(K^*),t,\zt).
$$
Therefore 
\begin{eqnarray}\label{inside}
\eta\in D(\zt, \delta_t)\cap\overline{D}_{t} \ \Rightarrow
\tilde{\Psi}_t(\eta)=[U^{t,\zt},\,\phi^{t,\zt}_{f}: f\in H(\OM)]_{\Psi_t(\eta)}.
\end{eqnarray}
Let us now fix $\zt\in \overline{D}_t$. Set $A(t,\zt):= D(\zt,r)\cap \overline{D}_t$. 
Due to connectedness, for any $x\in A(t,\zt)$, there
is a finite chain of discs $\Delta_1,\Delta_2,\dots,\Delta_{N(x)}$ of radius $\delta_t$ 
and points $\zt = y_0,\,y_1,\dots,y_{N(x)}=x$ such that 
$$
y_j=\text{centre of}\;\Delta_{j+1}, \qquad
y_{j+1}\in \Delta_{j+1}\cap A(t,\zt) \;\;\text{for $j=0,1,\dots,N(x)-1$.}
$$
Suppose we have been able to show that for any $x\in A(t,\zt)$ that can be 
linked to $\zt$ by a chain of at most $M$ discs in the above manner, we have
$$
\tilde{\Psi}_t(x)=[U^{t,\zt},\,\phi^{t,\zt}_{f}: f\in H(\OM)]_{\Psi_t(x)}.
$$ 
Now let $z\in A(t,\zt)$ be a point that is linked to $\zt$ by a chain 
$\Delta_1,\Delta_2,\dots,\Delta_{M+1}$ of discs of radius $\delta_t$ in the
above manner. In keeping with the above notation, let $y_M$ be the 
centre of $\Delta_{M+1}$. By our inductive hypothesis
$$
\tilde{\Psi}_t(y_M)=[U^{t,\zt},\,\phi^{t,\zt}_{f}:f\in H(\OM)]_{\Psi_t(y_M)}.
$$
By the representation (\ref{repn.}) and the fact that $U:=U^{t,\zt}\cap U^{t,y_M}$ is convex (see
the remarks following Fact~1),
\begin{eqnarray}\label{equal}
\phi^{t,\zt}_{f}|_{U}\equiv \phi^{t,y_M}_{f}|_{U} \; \text{for all} \; f\in H(\OM).
\end{eqnarray}
Now applying (\ref{inside}) to the pair $(y_M,z)$, we get 
\begin{align}
 \tilde{\Psi}_t(z) &= [U^{t,y_M},\,\phi^{t,y_M}_{f}: f\in H(\OM)]_{\Psi_t(z)} \notag \\
		   &= [U,\,\phi^{t,\zt}_{f}: f\in H(\OM)]_{\Psi_t(z)}, \notag
\end{align}
which follows from~(\ref{equal}) and the fact that, as $|z-y_M|<\delta_t\leq \delta$, 
$\Psi_t(z)\in U$. From the last equality, we get 
$\tilde{\Psi}_t(z)=[U^{t,\zt},\,\phi^{t,\zt}_{f}: f\in H(\OM)]_{\Psi_t(z)}$. Given 
(\ref{inside}), mathematical induction tells us that we have actually established the following:
\medskip

{\addtolength{\leftskip}{6mm}
\noindent{{\bf Fact 2:} For any $t\in S$ and $\zt\in \overline{D}_t$,
\begin{eqnarray}\label{agree} 
\phi^{t,\zt}_f|_{X(\zt,\eta)} \equiv \phi^{t,\eta}_f|_{X(\zt,\eta)}\;\text{for all $f\in H(\OM)$}, 
\end{eqnarray}
and for any $\eta\in \overline{D}_t$ such that $|\eta-\zt|<r$, 
where $X(\zt,\eta)= U^{t,\zt}\cap U^{t,\eta}$.}

}
\medskip

Now we have that $\{\overline{D}_t\}_{t\in[0,1]}$ is a continuous family (hence a uniformly continuous 
family) with respect to the Hausdorff metric. From uniform continuity, the fact that $[0,\tau)\subset S$, 
and from the properties of $\partial_*\Gamma$, we can find a $t_0\in S$ such that 
$(\tau-t_0)\leq \newfrc{r}{4}$ and such that
$$
 \overline{D}_\tau\subset B(\overline{D}_{t_0},\newfrc{r}{4}), \qquad
 \bdy{D}_\tau\subset B(\bdy{D}_{t_0},\newfrc{r}{4}),
$$
where, given a set $E\subseteq \cplx$ and $C>0$,  $B(E,C):= \cup_{\zt\in E}D(\zt,C)$. 
Let $\zt\in\overline{D}_\tau$. Then there exists $x(\zt)\in\overline{D}_{t_0}$ such that 
$|\zt-x(\zt)|<\newfrc{r}{4}$. Define
$$
\tilde{\Psi}_\tau(\zt):=[U^{t_0,x(\zt)},\phi^{t_0,x(\zt)}_f:f\in H(\OM)]_{\Psi_\tau(\zt)}.
$$
As $r\leq \delta$, it follows from Fact\,1 and (\ref{cont}) that 
$\tilde\Psi_\tau(\zt)\in\widetilde{\OM}$. 
\smallskip

\noindent{{\bf Claim:} {\em The function $\tilde\Psi_\tau:\overline{D}_\tau\rightarrow\widetilde{\OM}$ 
is well-defined and continuous.}}

\noindent{{\em Proof of claim:} Let $\zt\in\overline{D}_\tau$. Suppose $x(\zt), \ y(\zt)$ are 
two different points in $\overline{D}_{t_0}$ such that 
$|\zt-x(\zt)|<\newfrc{r}{4}$, $|\zt-y(\zt)|<\newfrc{r}{4}$. 
Then $|x(\zt)-y(\zt)|<\newfrc{r}{2}$. As $(\tau-t_0)\leq\newfrc{r}{4}$, it follows 
from (\ref{cont}) that $\Psi_\tau(\zt)\in U^{t_0,x(\zt)}\cap U^{t_0,y(\zt)}$. Therefore, 
by $(\ref{agree})$ 
$$
[U^{t_0,x(\zt)},\,\phi^{t_0,x(\zt)}_f:f\in H(\OM)]_{\Psi_\tau(\zt)}=
[U^{t_0,y(\zt)},\,\phi^{t_0,y(\zt)}_f:f\in H(\OM)]_{\Psi_\tau(\zt)}.
$$
Therefore $\tilde{\Psi}_\tau$ is well-defined.}
\smallskip

We now pick a $\zt\in\overline{D}_\tau$ and fix it. By the construction of the topology on 
$\widetilde{\OM}$, given any neighbourhood $G$ of $\tilde{\Psi}_\tau(\zt)$ in 
$\widetilde{\OM}$, there exists $\epsilon\in(0, d(K^*))$, sufficiently small, such that 
$\mathcal{N}(\epsilon,\tau,\zt)\subset G$. Now let $0<\epsilon<d(K^*)$. Since 
$\Psi_\tau:\overline{D}_\tau\rightarrow\Cn$ is continuous there exists 
$\sigma=\sigma(\epsilon)$, with $\sigma(\epsilon)\in(0,\newfrc{r}{4})$, such that 
if $|\eta-\zt|<\sigma, \ \eta\in\overline{D}_\tau$, then 
$\Psi_\tau(\eta)\in D^n(\Psi_\tau(\zt), \epsilon)$. Let 
$\eta\in\overline{D}_\tau$ such that $|\eta-\zt|<\sigma$. There exists $x(\eta)\in\overline{D}_{t_0}$ 
such that $|\eta-x(\eta)|<\newfrc{r}{4}$. Therefore $|x(\zt)-x(\eta)|<\newfrc{3r}{4}$.
Recall that $(\tau-t_0)\leq \newfrc{r}{4}$. Therefore, applying (\ref{cont}) and (\ref{agree}) 
we get $[U^{t_0,x(\zt)},\phi_f^{t_0,x(\zt)}:f\in H(\OM)]_{\Psi_{\tau}(\eta)}=
[U^{t_0,x(\eta)},\phi_f^{t_0,x(\eta)}:f\in H(\OM)]_{\Psi_{\tau}(\eta)}$. Therefore 
$\tilde{\Psi}_\tau(\eta)\in\mathcal{N}(\epsilon,\tau,\zt)$ whenever 
$|\zt-\eta|<\sigma$ and $\eta\in\overline{D}_\tau$. By the remarks at the begining of this 
paragraph, $\tilde{\Psi}_\tau$ is continuous at $\zt$. As $\zt$ is an arbitrary point of 
$\overline{D}_\tau$ we have that $\tilde{\Psi}_\tau$ is continuous. Hence the claim.
\smallskip

By definition of $\tilde\Psi_\tau$, $p\circ\tilde\Psi_\tau\equiv\Psi_\tau$. For all 
$x\in\bdy{D_{t_0}}$, $\tilde\Psi_{t_0}(x)=[\OM,f:f\in H(\OM)]_{\Psi_{t_0}(x)}$. 
Since $\bdy{D}_\tau\subset B(\bdy{D}_{t_0},\newfrc{r}{4})$, for any $\zt\in 
\bdy{D}_\tau$, there exists an $x(\zt)\in \bdy{D_{t_0}}$ such that $\zt\in D(x(\zt),\newfrc{r}{4})$. 
By the argument in the first part of the previous claim,
$\tilde\Psi_{\tau}(\zt)=[\OM,f:f\in H(\OM)]_{\Psi_{\tau}(\zt)}$ for all 
$\zt\in\bdy{D_{\tau}}$. This implies that $\tau\in S$.
\smallskip

We have established that $S$ is closed.
\smallskip

The above method for showing that $\tau\in S$ tells us more. As $[0,1]\ni t\longmapsto \overline{D}_t$  
is uniformly continuous relative to the Hausdorff metric, and as $\partial_*\Gamma$ is smooth,
there exists a constant $\delta^\prime>0$ such that
\begin{align}
 |t-s|<\delta^\prime \ &\Rightarrow \ \mathcal{H}(\overline{D}_t,\overline{D}_s)<\newfrc{r}{4}, \notag \\
 0<(t-s)<\delta^\prime \ &\Rightarrow \ \bdy{D}_t \in B(\bdy{D}_s,\newfrc{r}{4}), \notag
\end{align}
where $\mathcal{H}$ denotes the Hausdorff metric and $r$ is exactly as fixed above.
Let $\delta^*:=\min(\delta^\prime,\newfrc{r}{4})$. 
The same argument, with appropriate replacements where necessary, shows that if 
$t_0\in S$, then $\tau\in S$ for each $\tau\in[t_0, \text{min}(1, t_0+\delta^*))$.
Therefore $S$ is both open and closed. We have shown that $0\in S$. Hence $S=[0,1]$, which 
completes the proof.
\end {proof}
\medskip

\section{The proof of the main theorems}\label{S:proof}
\begin{proof}[The proof of Theorem~\ref{T:MainThm}]
By Proposition~\ref{KSTZ3} we have that there exists 
$\tilde{\Psi}_1:\overline{D}_1\rightarrow\widetilde{\OM}$, a continuous 
map such that $p\circ\tilde{\Psi}_1\equiv\Psi_1$ and 
$\tilde{\Psi}_1(\bdy D_1)\subset K^*$. The condition $p\circ\tilde{\Psi}_1\equiv\Psi_1$ 
actually implies that $\tilde{\Psi}_1\in H(D_1)$, by the definition of holomorphicity of maps 
with values in an unramified domain spread over $\Cn$. Thus, as 
argued in the proof of Proposition~\ref{KSTZ3},  
$\tilde{\Psi}_1(\overline{D}_1)\subseteq(\widehat{K^*})_{\widetilde{\OM}}$.
Therefore, by Result~\ref{Car-Thu}, 
$P(\tilde{\Psi}_1(\zt),d(K^*))\subseteq\widetilde{\OM}$ for all $\zt\in\overline{D}_1$. 
Let us define for any $r>0$,
$$
\omega(r):=\{\zt\in D_1:{\text{dist}}[\zt,\bdy D_1]>r\}.
$$

\noindent{{\bf Claim:} {\em There exists $\ent\in(0,d_{\OM}(K)/4)$ such that if 
$\zt,\eta\in\overline{D}_1$, and $D^{n}(\Psi_1(\zt),\ent)\cap D^{n}(\Psi_1(\eta),\ent)\neq\emptyset$, 
then 
\begin{align}
 \text{either} \; \ & D^{n}(\Psi_1(\zt),\ent)\cup D^{n}(\Psi_1(\eta),\ent)\subset 
 \cup_{z\in \Psi_1(\partial D_1)}D^n(z,d_{\OM}(K))\subset \OM, \notag \\
 \text{or} \; \ & (D^{n}(\Psi_1(\zt),2\ent)\cap D^{n}(\Psi_1(\eta),2\ent) 
 \cap\Psi_1(\overline{D}_1)\neq\emptyset \ \text{and} \notag \\ 
 & D^{n}(\Psi_1(x),2\ent)\cap \Psi_1(\overline{D}_1), \ \text{$x=\zt,\eta$, is path-connected.}) 
 \label{2nd}
\end{align}}}
\noindent{{\em Proof of claim:} Let us fix an 
$\tilde{\epsilon}\in(0,d_{\OM}(K))$, whence 
$D^{n}(z,\tilde{\epsilon})\subseteq\OM$ for all $z\in K$. As 
$\Psi_1$ is continuous on $\overline{D}_1$,
which is a compact set, there exists $\delta>0$ such that 
$\Psi_1(\overline{D}_1\setminus\omega(2\delta))\subset
\cup_{z\in \Psi_1(\partial D_1)}D^{n}(z,\newfrc{\tilde{\epsilon}}{4})$.
By construction we see that for any $\epsilon\in(0,\newfrc{\tilde{\epsilon}}{4}]$,
if $\zt\in\overline{D}_1\setminus\omega(2\delta)$ and $\eta\in\overline{D}_1$,
then
\begin{multline}\label{union}
D^{n}(\Psi_1(\zt),\epsilon)\cap D^{n}(\Psi_1(\eta),\epsilon)\neq\emptyset \\
\implies\,D^{n} (\Psi_1(\zt),\epsilon)\cup D^{n}(\Psi_1(\eta),\epsilon)\subset
\cup_{z\in \Psi_1(\partial D_1)}D^n(z,d_{\OM}(K)) \subset \OM.
\end{multline}
The above statement remains true when $\zt$ and $\eta$ are interchanged.}
\smallskip

Let us view $\Cn$ as a Hermitian manifold with $T(\Cn)$ as the holomorphic 
tangent space equipped with the standard Hermitian inner product $\langle\bdot\,,\bdot\rangle_{std}$
on each $T_{p}\Cn$,\,\,$p\in\Cn$. Let
$$
N_{\Psi_1}:= {\text{the normal bundle of}}\,\,\Psi_1(D_1)\,\, 
{\text{with respect to}}\,\,\langle\bdot\,,\bdot\rangle_{std}.
$$
Let $\pi$ denote the bundle projection. It is well known that for any relatively compact subdomain 
$\Delta\Subset D_1$, there 
exists 
$r_\Delta>0$, such that, if we define:
$$
N(\Psi_1,r_\Delta):=
\bigcup\nolimits_{p\in\Psi_1(\Delta)}\{v\in \left.N_{\Psi_1}\right|_p:
\langle v,v\rangle_{std}\,<r_{\Delta}^2\},
$$
then the map $\Theta(v):=\pi(v)+ v$ is a holomorphic imbedding of 
$N(\Psi_1,r_\Delta)$ into $\Cn$.
\smallskip 

Write $\Delta:=\omega(\newfrc{\delta}{2})$ and $\omega:=\omega(\delta)$. 
Since $\Psi_1(\overline{\omega})$ is a compact 
subset of $\Cn$, there exists an $\ent\in (0,\newfrc{\tilde{\epsilon}}{4})$ so small that:
\begin{itemize}
 \item $D^{n}(z,\ent)\cap(\cup_{w\in K}D^{n}(w,\ent))=\emptyset$ for all $z\in 
 \Psi_1(\overline{\omega})$;
 \item $D^{n}(z,2\ent)\cap \Psi_1(\overline{D}_1)$ is path-connected  
 for all $z\in \Psi_1(\Delta)$;
 \item $D^{n}(z,2\ent)\cap\Psi_1(\Delta)\Subset \Psi_1(\Delta)$ (with respect 
 to the relative topology on $\Psi_1(\Delta)$) for all $z\in\Psi_1(\overline{\omega})$;
 \item $D^{n}(z,\ent)\subseteq\Theta(N(\Psi_1,r_\Delta))$ for all 
 $z\in\Psi_1(\overline{\omega})$; and
 \item $\pi\circ\Theta^{-1}(z)\in D^{n}(p,2\ent)\cap\Psi_1(\Delta)$ for all
 $z\in D^{n}(p,\ent)$ and for all $p\in\Psi_1(\overline{\omega})$
\end{itemize}
(recall that $\Theta:N(\Psi_1,r_\Delta)\rightarrow\Cn$). The first condition
relies on our assumption on $K$.
\smallskip

Now suppose $\zt\neq \eta\in\Psi_1(\overline{\omega})$ and 
$D^{n}(\zt,\ent)\cap D^{n}(\eta,\ent)\neq\emptyset$. Let 
$w\in D^{n}(\zt,\ent)\cap D^{n}(\eta,\ent)$. Then, as $\Theta$ is an imbedding,
there is a unique $v\in N(\Psi_1,r_\Delta)$ such that $\Theta(v)=w$. By our 
construction 
$\pi(v)=\pi\circ\Theta^{-1}(w)\in D^{n}(\zt,2\ent)\cap D^{n}(\eta,2\ent)\cap \Psi_1(\Delta)$. 
Combining this with the statement culminating in (\ref{union}), we have the claim.
\smallskip

Write:
$$
 V:=\left(\cup_{\zt\in\overline{D}_1}D^{n}(\Psi_1(\zt),\ent)\right) 
 \cup\left(\cup_{z\in K}D^{n}(z,\ent)\right),\quad
 W:=\cup_{z\in K}D^{n}(z,\ent).
$$
Clearly, $W\subset V\cap\OM$ as $\ent<d_{\OM}(K)$ and $V\nsubseteq\OM$ 
as $\Psi_1(\overline{D}_1)\subseteq V$. Let us use the notation 
(\ref{repn.}) to represent $\tilde{\Psi}_1(\zt)$ (which is an $H(\OM)$-germ). 
We have established in the proof of Proposition~\ref{KSTZ3} --- see Fact\,1 
stated in Section~\ref{S:funda} --- that there is a representative 
$(U^{1,\zt},\,\phi^{\zt}_{f}:f\in H(\OM))$ of $\tilde{\Psi}_1(\zt)$ such that 
$U^{1,\zt}=D^{n}(\Psi_1(\zt),d(K^*))$. Given any $f\in H(\OM)$, define $G_{f}:V\rightarrow\cplx$ by
$$
 G_{f}(z) := \begin{cases}
		\phi^{\zt}_{f}(z), &\text{if $z\in D^{n}(\Psi_1(\zt),\ent)$}, \\
		f(z), &\text{if $z\in D^n(\theta,\ent)$ for some
		$\theta\in K\setminus\Psi_1(\overline{D}_1)$}.
		\end{cases}
$$
Here, for simplicity of notation, we write $\phi^{\zt}_{f}:=\phi^{1,\zt}_{f}$.
\smallskip

We shall prove that the function $G_f$ is well-defined on $V$. To start with, let us assume 
that for some $z\in V$, $z\in D^{n}(\Psi_1(\zt),\ent)\cap D^{n}(\Psi_1(\eta),\ent)$, where 
$\zt\neq\eta$. We make use of the claim above. We see from the proof of this claim that if any
{\em one} of $\zt$ or $\eta$ is in $\overline{D}_1\setminus \omega(2\delta)$, then the other point
is sufficiently close to (or on) $\partial D_1$ to imply, by  Proposition~\ref{KSTZ3}, that
$\phi^{\zt}_{f}=f=\phi^{\eta}_{f}$. By the proof of the above claim, we see that it remains
to consider those $(\zt,\eta)$ that lead to the outcome (\ref{2nd}) above. In this situation, there exists 
$\tau\in D_1$ such that $\Psi_1(\tau)\in D^{n}(\Psi_1(\zt),2\ent)\cap D^{n}(\Psi_1(\eta),2\ent)$. 
The {\em same} argument that leads to Fact\,2 stated in the proof of Proposition~\ref{KSTZ3} can be 
used to show the following analogue of Fact\,2:

{\addtolength{\leftskip}{6mm}
\noindent{{\bf Fact $\boldsymbol{2^\prime}$:} 
For any $t\in [0,1]$ and $\zt\in \overline{D}_t$, if $A_\zt:=$ the 
path-component of $\Psi^{-1}_t(U^{t,\zt})$ containing $\zt$, then
$$ 
 \phi^{t,\zt}_f|_{X(\zt,\eta)} \equiv \phi^{t,\eta}_f|_{X(\zt,\eta)} \; \text{for all $f\in H(\OM)$ 
 and for all $\eta\in A_\zt$}, 
$$
where $X(\zt,\eta)= U^{t,\zt}\cap U^{t,\eta}$.}

}
\smallskip

\noindent{Since $2\ent < d_{\OM}(K) \leq d(K^*)$, $\Psi_1(\tau)\in U^{1,x}, \ x=\zt,\eta$.
Furthermore, under the condition (\ref{2nd}), we see that 
$$
\tau \ \ \text{lies in the connectected component of $\Psi^{-1}_1(U^{1,x})$ containing $x$, 
$x=\zt,\eta$},
$$
since $\Psi_1$ is an imbedding. From Fact\,$2^\prime$, we deduce that 
$$
 \phi^{\zt}_f|_{U^{1,\zt}\cap U^{1,\tau}} \equiv \phi^{\tau}_f|_{U^{1,\zt}\cap U^{1,\tau}}, \qquad
 \phi^{\eta}_f|_{U^{1,\eta}\cap U^{1,\tau}} \equiv \phi^{\tau}_f|_{U^{1,\eta}\cap U^{1,\tau}},
$$
for all $f\in H(\OM)$. Therefore there is a neighbourhood $N$ of $\Psi_1(\tau)$ such that
\begin{equation}\label{rest}
\phi^{\zt}_{f}|_N \equiv \phi^{\eta}_{f}|_N.
\end{equation}
As $X:=D^{n}(\Psi_1(\zt),2\ent)\cap D^{n}(\Psi_1(\eta),2\ent)$ is 
connected, by (\ref{rest}) we have $\phi^{\zt}_{f}|_X\equiv
\phi^{\eta}_{f}|_X$. Therefore $\phi^{\zt}_{f}(z)=\phi^{\eta}_{f}(z)$.}
\smallskip

Note that there is no problem if $z\in D^{n}(\theta_1,\ent)\cap D^{n}(\theta_2,\ent)$, 
$\theta_1\neq \theta_2\in K\setminus \Psi_1(\overline{D}_1)$. Hence, we consider the case
when $z\in D^{n}(\Psi_1(\zt),\ent)\cap D^{n}(\theta,\ent)$, where $\zt\in\overline{D}_1$ and $\theta\in K$.
In this case, the first condition among those determining $\ent$ implies that
$\zt\in \overline{D}_1\setminus \omega(\delta)$. Thus, the argument for well-definedness is
{\em exactly} as in the first few lines of the last paragraph. This completes the proof of well-definedness.
\smallskip

By the construction of $G_f$ on $V$, we have, given that $z\in D^{n}(\theta,\ent)$ 
for some $\theta\in K$, $G_{f}(z)=f(z)$. Therefore $G_{f}|_{W}\equiv f|_{W}$. Finally, by the fact 
that holomorphicity is a local property, we conclude that $G_{f}\in H(V)$.
\end{proof}

We now turn to the proof of Theorem~\ref{T:thm2}. For this purpose, we need a result by 
Jupiter.

\begin{result}[Jupiter]\label{jupiter}
Let $\OM$ be a domain in $\mathbb{C}^{2}$, and let $(\widetilde{\OM},p,j)$ be 
its envelope of holomorphy. If $p$ is injective on $p^{-1}(\OM)$, then $p$ 
is injective on $\widetilde{\OM}$. In other words, if $\widetilde{\OM}$ is 
schlicht over $\OM$, then $\widetilde{\OM}$ is schlicht.
\end{result}

\begin{proof}[The proof of Theorem~\ref{T:thm2}]
By Result~\ref{jupiter}, $p:\widetilde{\OM}\rightarrow \CC$ is an injective map.
As $p$ is a local biholomorphism, $p(\widetilde{\OM})$ is an open connected set in $\CC$ and
$p^{-1}$ is a holomorphic map on it. Let us write 
$$
 V:=\bigcup\nolimits_{z\in \Psi_1(\overline{D}_1)}D^n(z,d_{\OM}(K))\cup \OM.
$$
We have $V\nsubseteq\OM$, since $\Psi_1(\overline{D}_1)\nsubseteq\OM$.
\smallskip

As $d_{\OM}(K)\leq d(K^*)$ and $\tilde{\Psi}_1(\overline{D}_1)\in 
(\widehat{K^*})_{\widetilde{\OM}}$, the Cartan--Thullen Theorem gives 
$V\subset p(\widetilde{\OM})$. Thus, if for any $f\in H(\OM)$, we define $G_{f}:V\rightarrow\cplx$ 
by
$$
G_f:=F_{f}\circ p^{-1},
$$
then $G_f\in H(V)$ and $G_{f}|_{\OM}\equiv f$.
\end{proof}

\begin{remark}
Note that we could have taken $V=p(\widetilde{\OM})$ in the above proof. But we want to 
highlight the fact that the Kontinuit{\"a}tssatz is {\em a means of deducing information 
about $p(\widetilde{\OM})$ when very little is known about the geometry of $\widetilde{\OM}$.} 
The spirit of Theorem~\ref{T:thm2} is that, while we know very little about the finer  
geometric properties of $\widetilde{\OM}$, this theorem shows what we can learn about 
$p(\widetilde{\OM})$ if we have a little extra information about $(\widetilde{\OM},p,j)$.
In applications of Theorem~\ref{T:thm2} and Theorem~\ref{T:MainThm}, it is understood that we 
have good information about the pair $(\OM,K)$. It is in keeping with this spirit that we 
chose $V$ as above.
\end{remark}  

\noindent{\textbf{Acknowledgements.} The author is grateful to Gautam Bharali and 
Kaushal Verma for their valuable comments and for many useful discussions during the course
of this paper.}

\end{document}